\documentclass{amsart}

\usepackage{epsfig,graphicx,amsfonts,amssymb,amsmath,amscd,amsthm,color,stmaryrd}
\usepackage[all]{xy}
\usepackage{url}
\usepackage{latexsym}
\usepackage{pgf,tikz}
\usetikzlibrary{arrows}
\usepackage[colorlinks=true]{hyperref}

\numberwithin{equation}{section}

\newtheorem{thm}[subsubsection]{Theorem}
\newtheorem{prop}[subsubsection]{Proposition}
\newtheorem{lem}[subsubsection]{Lemma}

\theoremstyle{definition}
\newtheorem{defin}[subsubsection]{Definition}
\newtheorem{rem}[subsubsection]{Remark}

\def\into{{\hookrightarrow}}
\def\oU{{\overline{U}}}
\def\oY{{\overline{Y}}}

\def\oC{{\overline{C}}}
\newcommand{\tilOmega}{{\widetilde{\Omega}}}
\newcommand{\tilomega}{{\widetilde{\omega}}}
\newcommand{\tileta}{{\widetilde{\eta}}}
\newcommand{\tilc}{{\widetilde{c}}}
\newcommand{\tilt}{{\widetilde{t}}}
\newcommand{\tila}{{\widetilde{a}}}

\newcommand{\tilf}{{\widetilde{f}}}

\newcommand{\tilk}{{\widetilde{k}}}
\newcommand{\tilX}{{\widetilde{X}}}

\newcommand{\kcirc}{{k^\circ}}
\newcommand{\kcirccirc}{{k^{\circ\circ}}}
\def\.{{,\dotsc, }}

\newcommand{\val}{{\rm val}}
\newcommand{\chara}{{\rm char}}

\newcommand{\Spec}{{\rm Spec}}

\newcommand{\CCC}{{\mathbb C}}

\newcommand{\RR}{{\mathbb R}}
\newcommand{\ZZ}{{\mathbb Z}}
\newcommand{\PP}{{\mathbb P}}

\newcommand{\GG}{{\mathbb G}}

\newcommand{\CA}{{\mathcal A}}

\newcommand{\CO}{{\mathcal O}}

\newcommand{\CM}{{\mathcal M}}
\newcommand{\oCM}{\overline{\mathcal M}}

\newcommand{\CT}{{\mathcal T}}

\newcommand{\CY}{{\mathcal Y}}

\newcommand{\CC}{{\mathcal C}}

\newcommand{\ft}{{\mathfrak t}}
\newcommand{\fh}{{\mathfrak h}}

\newcommand{\gr}{{\rm gr}}

\newcommand{\Res}{{\rm Res}}
\newcommand{\res}{{\rm res}}

\newcommand{\ord}{{\rm ord}}
\newcommand{\Star}{{\rm Star}}

\newcommand{\op}{{\rm op}}
\newcommand{\divisor}{{\rm div}}

\def\:{\colon}
\def\Br{{\rm Br}}
\def\val{\nu}

\def\b1{\boldsymbol 1}

\def\ko{{k^\circ}}

\begin{document}

\title[Reduction and lifting of differential forms]{Reduction and lifting problem for differential forms on Berkovich curves}
\author{Michael Temkin and Ilya Tyomkin}

\thanks{M.T. was supported by ERC Consolidator Grant 770922 - BirNonArchGeom, I.T. was supported by ISF grant 821/16}

\address{Einstein Institute of Mathematics, The Hebrew University of Jerusalem, Giv'at Ram, Jerusalem, 91904, Israel}
\email{michael.temkin@mail.huji.ac.il}
\address{Department of Mathematics, Ben-Gurion University of the Negev, P.O.Box 653, Be'er Sheva, 84105, Israel}
\email{tyomkin@math.bgu.ac.il}

\begin{abstract}
Given a complete real-valued field $k$ of residue characteristic zero, we study properties of a differential form $\omega$ on a smooth proper $k$-analytic curve $X$. In particular, we associate to $(X,\omega)$ a natural tropical reduction datum combining tropical data of $(X,\omega)$ and algebra-geometric reduction data over the residue field $\tilk$. We show that this datum satisfies natural compatibility condition, and prove a lifting theorem asserting that any compatible tropical reduction datum lifts to an actual pair $(X,\omega)$. In particular, we obtain a short proof of the main result of \cite{BCGGM}.
\end{abstract}

\maketitle

\section{Introduction}

\subsection{Motivation}
The motivation for this work is two-fold. A general motivation is to establish one more important instance of tropical reduction on Berkovich curves, while a concrete motivation was to re-interpret in terms of non-archimedean geometry the main result of \cite{BCGGM}, which describes a compactification of a moduli space of curves with a differential form. In particular, we obtain a short proof of \cite[Theorem~1.3]{BCGGM} and suggest a natural interpretation of the global residue condition discovered in \cite{BCGGM}. Our methods are much more algebraic and apply over any complete valuation field $k$ of residual characteristic zero, though we do use analytic geometry over $k$ and algebraicity of proper $k$-analytic curves.

We learned about \cite{BCGGM} from the talks given by Sam Grushevsky at our departments, and we are grateful to Sam for that and for the subsequent discussions.

\subsection{Tropical reduction}
Let us explain first what we mean by tropical reduction of $k$-analytic objects.

\subsubsection{Metrized curve complexes}
Two main invariants of a non-archimedean field $k$ are its residue field $\tilk$ and the group of values $|k^\times|$. Algebraic geometry over $\tilk$ and polyhedral or tropical geometry over $|k^\times|$ are very useful in studying geometry over $k$. There are various constructions (often depending on choices) that associate to a $k$-analytic space $X$ reductions $\tilX$, which are $\tilk$-varieties, and skeletons $\Gamma_X$, which are piecewise linear spaces. In this paper, $X$ is a curve and $\Gamma_X$ is just a metric graph. Naturally, the two types of invariants are related and one may want to combine them into a single invariant of a mixed type. Amini and Baker introduced in \cite{Amini-Baker} metrized curve complexes by placing $\tilk$-curves as vertices of $\Gamma_X$ and identifying the ends of the edges with points of these curves, and we will follow this approach in the paper. We will use the term {\em tropical reduction} to denote constructions mixing reduction with tropical data.

\begin{rem}
A more conceptual way to encode the same information is to enrich a nodal reduction $\tilX$ with tropical information. This is achieved by providing $\tilX$ with a natural log structure induced from the formal model. In fact, this makes $\tilX$ a log smooth curve over $\tilk$ provided with the log structure $(\kcirc\setminus\{0\})/(1+\kcirccirc)$, see \cite[\S5.1.8]{BT20}. It is natural to call such construction {\em log reduction}, as in \cite{BT20}. However, we do not use log geometry in the sequel, and the additional data we consider is still an ad hoc mix of tropical and algebra-geometric data. It is an interesting and important question if it can be interpreted as a meaningful object of log geometry.
\end{rem}

\subsubsection{$p$-covers}
It is natural to ask if tropical reduction of a $k$-analytic curve $X$ extends to finer objects. In \cite{BT20} this problem was solved for finite covers $Y\to X$ of degree $p=\chara(\tilk)$. The reduction data was rather non-trivial, and involved in addition to a finite map of tropical reduction data also a pl function on $\Gamma_Y$ (the different function) and certain meromorphic differential forms at the vertex curves of $\Gamma_Y$. Moreover, \cite{BT20} provides a list of compatibilities between these objects and its main result is a lifting theorem claiming that any such tropical reduction datum can be lifted to an actual $p$-cover $Y\to X$. This indicates that the constructed reduction is ``as fine as possible'' and no ``hidden parameters are left''.

\subsubsection{Differential forms}
The goal of this paper is to solve a similar problem for $X$ provided with a global meromorphic differential form $\omega$: construct a tropical reduction of $(X,\omega)$, describe a list of compatibilities this datum satisfies, and show that the list is complete by proving a lifting theorem -- any compatible tropical reduction datum lifts to an actual pair $(X,\omega)$.

\subsection{Compactifications of $\Omega\CM_g$}\label{compactsec}
Next, we briefly recall some results of \cite{BCGGM}. For simplicity we consider regular differential forms, the meromorphic case is similar, but requires a minimal care for poles. Pairs consisting of a smooth proper curve $X$ with a non-zero regular differential form $\omega\in\Gamma(\Omega_X)$ defined up to a scalar are classified by the projectivized Hodge bundle $\PP\Omega\CM_g$ whose fibers over $\CM_g$ are $\PP^{g-1}$. For a partition  $\mu=(\mu_1\.\mu_n)$ of $2g-2$, let $\PP\Omega\CM_g(\mu)$ be the stratum of pairs $(X,\omega)$ such that the pattern of zeros of $\omega$ is $\mu$. Since $\omega$ is determined up to a scalar by its divisor, the stratum embeds also into the space $\CM_{g,n}$ (modulo the group of symmetries of the partition).

The two embeddings induce two different compactifications: by taking the closure in the natural extension of the projectivized Hodge bundle to $\PP\Omega\oCM_g\to\oCM_g$, where $\Omega$ is replaced by the dualizing sheaf over the nodal curves, and by taking the closure in $\oCM_{g,n}$.  Moreover, it is even more natural to consider the finer compactification $K\CM_{g,\mu}$ refining these two, which is obtained by taking the closure in $\PP\Omega\oCM_{g,n}$. In a sense, one modifies the first compactification by imposing the condition that the pattern of zeros is preserved, i.e., stay in the smooth locus and do not collide. Points of $\PP\Omega\oCM_{g,n}$ are given by nodal curves $X$, smooth points $p_1\.p_n\in X$ and a non-zero section $\omega\in\Gamma(\omega_X)$ up to a multiplicative scalar.

The main result of \cite{BCGGM} describes which points $(X,p,\omega)$ belong to $K\CM_{g,\mu}$. The starting point is that the datum $(X,p,\omega)$ is still not fine enough since $\omega$ may vanish on components of $X$. If $(X,p,\omega)$ is the central fiber over a complex disc and the generic fiber $(X_t,p_t,\omega_t)$ is smooth, then the order of pole $\ell_i$ of $\omega$ along the components $\tilX_i$ of $X$ is a valuable invariant, and after rescaling by $t^{\ell_i}$ one obtains a non-trivial limit $\tilomega_i=t^{\ell_i}\omega|_{\tilX_i}$, which is a differential form on $\tilX_i$ that might have poles at the nodes. This idea goes back to Kontsevich, who also noted that this data satisfies some simple compatibility conditions along the nodes. The function $\ell$ associating $\ell_i$ to $\tilX_i$ is called {\em level function}.

It was proved in \cite{BCGGM} that the datum $(X,p,\ell,\tilomega)$ also satisfies a much less intuitive {\em global residue condition} taking into account the structure of the dual graph of the components of bounded level, and imposing extra vanishing conditions on the residues of $\tilomega_i$ at a given level. The main result \cite[Proposition~3.8]{BCGGM} shows that any datum satisfying the global residue condition extends to a family over a disc. As an immediate corollary one obtains the following characterization of $K\CM_{g,\mu}$: a tuple $(X,p,\omega)$ lies in $K\CM_{g,\mu}$ if and only if it lifts to a compatible data $(X,p,\ell,\tilomega)$ satisfying the global residue condition, see \cite[Theorem~1.3]{BCGGM}.

\begin{rem}
It is not an accident, that points of $K\CM_{g,\mu}$ are described in terms of liftability of a datum $(X,p,\omega)$ to a compatible datum $(X,p,\ell,\tilomega)$. In \cite{BCGGM2}, a finer compactification of $\PP\Omega\CM_{g}$ is constructed, whose points more or less correspond to the data $(X,p,\ell,\tilomega)$ with an additional piece of ``stacky" information, called {\em prong-matching} in {\em loc.cit.} The latter compactification is smooth and has a natural modular interpretation. Its construction is rather complicated, and we do not discuss it in this paper.
\end{rem}

Finally, let us say a few words about the methods of \cite{BCGGM}. The main result is proved over $\CCC$ in two ways. The first one uses plumbing techniques over the punched disc, and it has many similarities with our method. The other one relies on the theory of flat surfaces, and its output is a family over $[0,\varepsilon)$ rather than over a complex disc, it is not related to the sequel.

\subsection{Outline of the paper}
Now, let us explain what is done in this work. The case considered in \cite{BCGGM} corresponds to $k=\widehat{\cup_n\CCC(\!(t^{1/n})\!)}$.

\subsubsection{Tropical reduction of $\omega$}
We work with a nice $k$-analytic curve $X$ (see \S\ref{convsec}) and a meromorphic differential form $\omega$ on $X$. Also, we fix a skeleton $\Gamma$ of $X$ containing all poles and zeros of $\omega$, and we enrich $\Gamma$ to a metrized curve complex $\tilX=(\Gamma,C_x,p_i)$, see \S\ref{sec:metcc} for a precise definition. Note that a tropical reduction of a function $f$ with poles and zeros in $\Gamma$ is nothing but the pl function $\log|f|$ on $\Gamma$ and scaled reductions $\tilf_x$ of $f$ at the vertices $x\in \Gamma$ of type $2$. The latter are meromorphic functions on the reduction components $C_x$, satisfying simple compatibility conditions. This is a very natural construction, that appeared in different settings in \cite[Theorem~5.15]{BPR} and in \cite[\S2.2.1]{T12}.

In \S\ref{sec:dfoac} we introduce a tropical reduction of $\omega$. A natural metrization $\|\ \|$of $\Omega_X$ was constructed in \cite{Tem16}, so we can just imitate the case of a function: we associate to $\omega$ the level function $\ell=\log\|\omega\|$ on $\Gamma$ and scaled reductions $\tilomega_x$ at the vertices of type $2$. In fact, we prefer a slightly more technical but canonical approach -- we consider the so called {\em graded reduction}, so that $\tilomega_x$ lives at the $\ell(x)$-graded piece. This allows us to avoid the choice of a scaling element $c\in k^\times$. Note also that the reductions of $\Omega$ with respect to this norm were computed in \cite{BT20}: $\tilOmega_x=\Omega^{\rm log}_{C_x}$, where $C_x$ is the reduction of $X$ at $x$. So we view $\tilomega_x$ as a meromorphic differential form on $C_x$ but measure its zeros and poles using the log-order.

\begin{rem}
(i) A family $(X_t,\omega_t)$ over a punched complex disc gives rise to a pair $(X,\omega)$ over $k=\widehat{\cup_n\CCC(\!(t^{1/n})\!)}$, and then the tropical reduction datum of $(\ell,\tilomega)$ of $(X,\omega)$ is nothing but the datum $(\ell,\tilomega)$ of \cite{BCGGM}. Also, it was noted in \cite[Theorem~5.2]{MUW} that this construction extends to the case when $k$ is an arbitrary DVF.

(ii) In \cite{BCGGM} and \cite{MUW} one considers the usual order of zeros and poles of $\tilomega_x$, so the compatibility condition along an edge $e$ is that the sum of orders on the two sides of $e$ is $-2$. We use the log order, so in our case the sum is 0, similarly to reduction of functions or sections of any other line bundle.
\end{rem}

Finally, we introduce one more tropical reduction invariant -- {\em the residue function} $\Re_\omega$, which associates to an oriented edge of $\Gamma$ a scalar in ground field $k$. To do so, we show that $\omega$ has a well-defined residue $\Res_\CA(\omega)\in k$ along any oriented annulus $\CA$. This is a rather straightforward generalization of a classical result of Serre, see \cite[Ch. II, \S11]{Serre}. We define $\Re_\omega(e)$ to be $\Res_\CA(\omega)$, where $\CA$ is an oriented annulus whose skeleton is contained in $e$ and has compatible orientation. The compatibilities satisfied by the residue function are also very natural: it changes sign when the orientation of $e$ is changed, it is harmonic at the vertices of type $2$, if $e$ is a leg of $\Gamma$ and $p_l$ is the corresponding point of type $1$ then $\Re_\omega(l)=\res_{p_l}(\omega)$, and if $x\in\Gamma$ is a vertex of type $2$ and $e\in\Star(x)$ then the (graded) reduction of $\Re_\omega(e)$ coincides with the residue of $\tilomega_x$ at the point $p_e\in C_x$ corresponding to the branch $e$. It turns out that the global residue condition of \cite{BCGGM} is an immediate corollary of the above properties of $\Re_\omega$.

\begin{rem}
(i) The function $\Re_\omega\:E(\Gamma)\to k$ is a sort of a tropical datum, but it obtains values in $k$ rather than $\tilk$. We do not know other interesting examples of this form.

(ii) Our tropical reduction datum $(\tilX,\ell,\tilomega,\Re_\omega)$ extends the datum of \cite{BCGGM} by including the residue function $\Re_\omega$. The latter is a sort of a hidden parameter in {\em loc.cit.} responsible for the global residue condition -- the reduction of $\Re_\omega$ can vanish on components of high level, while the global harmonicity of $\Re_\omega$ might force non-trivial restrictions at lower levels.
\end{rem}

\subsubsection{Main results}
Our main results concern with tropical reduction and lifting. In Theorem \ref{thm:tropicalization}, we prove that the tropical reduction $(\tilX,\ell,\tilomega,\Re_\omega)$ satisfies a list of natural compatibility conditions, some of which have been mentioned above. Conversely, Theorem~\ref{thm:lifting} asserts that any compatible datum $(\tilX,\ell,\tilomega,\Re)$ lifts to a pair $(X,\omega)$ over $k$. We shall emphasize that in this paper, we do not consider the question of liftability of purely tropical data, i.e., data consisting only of a tropical curve and a level function. This question was answered in \cite{MUW}, to which we refer the interested reader for details.

\begin{rem}
We want to stress that both $X$ and $\omega$ are constructed in the lifting theorem. This is analogous to the main result of \cite{BT20}, where one lifts an appropriate tropical reduction datum to a $p$-cover $Y\to X$, but neither $Y$ nor $X$ can be fixed. Notice that one cannot expect a compatible datum to be liftable with a prescribed lift $X$ of $\tilX$. Indeed, by \cite[Theorem~1]{KZ03}, there exist components of the moduli space of stratified differentials supported over the hyperelliptic locus of the moduli space of curves. Thus, the reduction $(\tilX,\ell,\tilomega,\Re_\omega)$ of a pair $(X,\omega)$ in such a component can not be lifted to a non-hyperelliptic lift $X'$ of $\tilX$.
\end{rem}

\subsubsection{The methods}
Our method of proving the reduction and lifting theorems for a differential form in the residual characteristic zero case is surprisingly close to the method used in \cite{BT20} to prove analogous results for $p$-covers in the positive residual characteristic case. Everything is based on Proposition~\ref{goodprop}, which classifies differential forms without zeros and poles on an annulus $\CA$. Explicitly, it asserts that there exists a {\em good coordinate} $t$ on $\CA$ such that $\omega$ acquires a binomial form $(c_nt^n+c_0)\frac{dt}t$. Moreover, $c_0$ is the residue of $\omega$ on $\CA$ and $|c_nt^n|=\|\omega\|$ on the skeleton $e_\CA$, so the isomorphism class of $\omega$ is determined by the two invariants: the residue and the norm on the skeleton. Note that this result is a natural generalization of a classical result providing a local description of meromorphic differential forms on Riemann surfaces. The proof is similar, and constructs the required coordinate by a series of approximations. Notice also, that a similar assertion in the complex-analytic setting plays an important role in the proof of the main result in \cite{BCGGM}.

\begin{rem}
The analogous result in \cite{BT20} shows that a $p$-cover $Y\to\CA$ is presented by $y=t^p+c_nt^n$ for an appropriate choice of coordinates on $Y$ and $\CA$, and $|c_nt^n|$ coincides with the different function $\delta_f$ on $e_\CA$. It is also proved by successive approximations, but the computation is subtler, see Theorem~4.3.8 and Corollary~4.3.9 in {\em loc.cit.}
\end{rem}

The most subtle result about the residue function is its harmonicity at a vertex $x\in\Gamma$ of type 2. We prove this by cutting out a star-shaped neighborhood $U$ of $x$, patching it by discs to a smooth proper curve $\oU$, and extending the form into the added discs using its description on the intersection annuli. In this way we obtain a full control on the poles of $\omega$ on $\oU$ and the harmonicity follows from the fact that the residues on the algebraic curve $\oU$ sum up to zero.

Finally, the lifting theorem is proved by lifting forms $\tilomega_x$ to star-shaped neighborhoods $U_x$ and gluing these to a single curve $X$ along annuli using the same classification of forms on the annuli. The only subtle point here is to lift the forms so that the residues at the edges at $v$ are precisely $\Re_\omega(e)\in k$, as an arbitrary lift of $\tilomega_x$ would only give a form with residues having the same graded reduction as $\Re_\omega(e)$'s. We solve this problem, again, by patching $U_x$ to a good reduction curve $\oU_x$ and proving algebraically that the problem of precise lift of residues is solvable for $\oU_x$.

\subsubsection{Torsor of good coordinates}
Existence of good coordinates was basic for all our proofs. We conclude this paper with a detailed study of the set of all such coordinates in \S\ref{sec5}. In particular, we prove that they form a torsor under the group $G_n(\ko)$, where $G_n:=\GG_m\ltimes_n \GG_a$ and $\GG_m$ acts on $\GG_a$ by $\lambda(\mu):=\lambda^{-n}\mu$, which gives rise to an identification of certain torsors in the tropical reduction. This result provides another interpretation of the prong matching of \cite{BCGGM2}, and can be useful for further studies.

\subsubsection{Comparisons}
Our method was inspired by that of \cite{BT20}. Despite the similarity, the lifting and tropical reduction for forms is a bit more straightforward and less technical, than in the case of $p$-covers, though the residue function $\Re_\omega$ provides a new aspect for this type of problems. As for \cite{BCGGM}, we are not familiar enough with the technique of plumbing, but it seems certain that there is a large intersection of ideas between the approaches. It still might be the case that the non-archimedean approach isolates the core problems one has to solve in a sharper way, whence leading to a shorter argument.

\subsubsection{Open questions}
In addition, our approach at least makes it reasonable to ask what happens without the restriction on the characteristic. At the very least, our construction of the tropical reduction datum works as well. To simplify the arguments, we use the assumption that $\chara(\tilk)=0$, but it can be removed. On the other hand, our classification of forms on annuli certainly does not hold when $\chara(\tilk)=p$, and we do not know what happens to the lifting theorem. Studying the latter case is a very interesting question for further research. It seems possible that in this case the tropical reduction datum is incomplete (there are some other tropical or reduction invariants), and the lifting theorem will hold only for a finer tropical reduction datum.

\subsection{Notation and convention}\label{convsec}
Throughout the paper $k$ denotes an algebraically closed complete non-archimedean field of residue characteristic $0$, $\val\:k\to \RR\cup\{\infty\}$ is the valuation, $\ko$ is the ring of integers, $\kcirccirc$ its maximal ideal, and $\tilde{k}$ is the residue field. When it is more convenient to use multiplicative notation we use the absolute value $|c|:=10^{-\val(c)}$.

By a {\em nice} $k$-analytic curve we mean a quasi-smooth connected compact separated strictly $k$-analytic curve. By a {\em star-shaped curve} we mean a pair $(X,x)$ where $X$ is a nice $k$-analytic curve, $x\in X$ a point of type two, and $X\setminus\{x\}$ a disjoint union of open discs and semi-open annuli.

A {\em branch} of a nice curve $X$ at a point $x$ is an equivalence class of germs of intervals $[x,y]\subset X$ and the set of all branches is denoted $\Br(x)$. In particular, there is one branch at a point of type one, and for a point $y$ of type 2 there is a one-to-one correspondence between the branches $e\in\Br(y)$ and the $\tilk$-points $p_e\in C_y$.

In this paper, the tropical curves are skeleta of analytic curves. Thus, we include the vertices at infinity that correspond to points of type $1$. Other vertices correspond to points of type $2$. We denote the set of vertices of type $i$ by $V_i(\Gamma)$. The edges adjacent to the vertices of type $1$ are called legs, and have infinite length. Other edges are bounded. All edges of tropical curves are {\em oriented}, and each bounded edge (even a loop) is considered twice by equipping it with the two possible orientations. The legs are always oriented towards the vertex of type $1$. By abuse of notation the set of oriented edges and legs is denoted by $E(\Gamma)$, and the set of legs by $L(\Gamma)$. If $e\in E(\Gamma)$ is bounded then $e^\op$ denotes the same edge with the opposite orientation. Finally, $\ft\:E(\Gamma)\to V(\Gamma)$ and $\fh\:E(\Gamma)\to V(\Gamma)$ denote the tail and the head functions respectively. If $\Gamma$ is a tropical curve and $x\in V(\Gamma)$ is of type $2$ then $\Star(x)$ denotes the set of oriented edges of $\Gamma$ with tail $x$. We shall not distinguish between tropical curves and their geometric realizations.

\section{Tropical reduction of curves with differentials}\label{sec:dfoac}
Let $X$ be a nice $k$-analytic curve equipped with a non-zero meromorphic differential form $\omega$. In this section we associate to $(X,\omega)$ a natural tropical and reduction datum.

\subsection{Skeletons of the pair}
By a skeleton of $(X,\omega)$ we mean any skeleton $\Gamma$ of $X$ such that $D=\divisor(\omega)\subset\Gamma$.

\begin{rem}
(i) It is well known that there exists such a minimal skeleton unless $X=\PP^1_k$ and $|D|\le 2$, or $X=E$ is a Tate curve and $D=0$. The first case happens when $\omega=cx^ndx$ for an appropriate choice of a coordinate $x$ on $\PP^1_k$ and the second case happens when $\omega$ is regular on $E$.

(ii) Our further constructions will work for any choice of $\Gamma$. It is natural to work with the minimal skeleton, if it exists, obtaining the stable tropical reduction.
\end{rem}

\subsection{The metrized curve complex with boundary}\label{sec:metcc}
The first ingredient of the tropicalization is a version of the metrized curve complex as introduced in \cite{Amini-Baker}.  For each vertex $x\in\Gamma$ of type $2$, let $C_x$ be the reduction curve over $\tilde{k}$ corresponding to $x$, and for each edge $e\in\Star(x)$, let $p_e\in C_x$ be the corresponding point. Then the datum $\mathcal{C}:=\left(\Gamma,(C_x)_{x\in V_2(\Gamma)},(p_e)_{e\in E(\Gamma)}\right)$ is the metrized curve complex associated to $(X,\Gamma)$, cf. \cite[Section~1.2]{Amini-Baker}. Notice that unlike in {\em loc.cit.}, in our setting some of the curves $C_x$ may be non-proper. In fact, $C_x$ is not proper if and only if $x$ is a boundary point of $X$. We shall call such a complex {\em metrized curve complex with boundary}, and the set of vertices $x\in V_2(\Gamma)$ for which $C_x$ is not projective will be called {\em the boundary} of $\mathcal{C}$.

\subsection{The norm and the graded reduction of $\omega$}\label{sec:gradedred}

\subsubsection{The level function}
Let $\|\;\|$ be the K\"ahler norm on the sheaf of differential forms as defined in \cite{Tem16}. It induces a {\em level function} on the tropical curve $\Gamma$ defined by $\ell(x):=\log\|\omega\|_x$. More explicitly, if $x\in\Gamma$ is of type two then $\ell(x)=\val(c)$, where $c\in k^\times$ is any scalar for which $c\omega$ admits a non-zero reduction, cf. \cite[\S5.3]{MUW}. By \cite[Theorem~8.1.6]{Tem16}, the level function $\ell$ is piecewise $\val(k^\times)$-integral affine and continuous. Hence $\ell$ is completely defined by the collection of its values at the finite vertices of $\Gamma$ together with the slopes of $\ell$ along the unbounded edges.

\subsubsection{The graded reduction}
For $x\in X$, denote by $\Omega_{X,x}^{\le r}$ the space of sections $\eta\in \Omega_{X,x}$ for which $\|\eta\|_x\le r$. Then $\Omega_{X,x}^{\le r}$ defines a filtration on $\Omega_{X,x}$, and we denote the associated graded $\tilde{k}(C_x)$-module by $\Omega_x^\gr$. We shall mention that each $c\in k^\times$ induces an isomorphism between the $|c|$-graded summand of $\Omega_x^\gr$ and $\Omega_{X,x}^{\circ}/\Omega_{X,x}^{\circ\circ}=\Omega_{\tilde{k}(C_x)/\tilde{k}}$, and in fact, $\Omega_x^\gr$ is canonically isomorphic to $\Omega_{\tilde{k}(C_x)/\tilde{k}}\otimes_{\tilde{k}}k^\gr$. Each section $\eta\in \Omega_{X,x}$ induces a uniquely defined graded element $\tilde{\eta}_x^\gr\in \Omega_x^\gr$ and a well-defined class $\tilde{\eta}_x\in \Omega_{\tilde{k}(C_x)/\tilde{k}}/\tilde{k}^\times$ of forms modulo scaling, which are called the {\em graded tropical reduction} and the {\em scaled reduction} of the form $\eta$ at the point $x$.

Let $\tileta_x^\gr$ be a graded form. Since the divisor of a form is independent of scaling, we have a well defined divisor $\divisor(\tileta_x^\gr):=\divisor(\tileta_x)$. We can also extend the notion of residue of a form, which will take values in $k^\gr$. To do so, pick any $c\in k^\times$ whose norm is equal to the grading of $\tileta_x^\gr$, and set $\res_p(\tileta^\gr_x):=c\cdot \res_p(c^{-1}\tileta^\gr_x)$. The result is clearly independent of $c$. The following proposition is a particular case of \cite[Lemma~3.3.2]{BT20}:

\begin{prop}\label{prop:gradedred}
For any $x\in V_2(\Gamma)$ and $e\in\Star(x)$, the slope $\frac{\partial\ell}{\partial e}$ of the level function $\ell$ along $e$ is equal to the minus log-order of $\tileta_x$ at $p_e$. In particular, $$\divisor(\tileta_x^\gr)=\sum_{e\in\Star(x)}\left(-1-\frac{\partial\ell}{\partial e}\right)p_e,$$ and if $x$ is not a boundary point of $X$ then $$\sum_{e\in\Star(x)}\frac{\partial\ell}{\partial e}=2-2g(C_x)-|\Star(x)|.$$
\end{prop}

\begin{rem}
In \cite{BCGGM} and \cite{MUW}, the reduction datum associated to a one-parameter degeneration of a differential form on a stable curve consists of a {\em level function} and a {\em twisted differential}. The level function is $\ell$ and the twisted differential is the scaled reductions $\tilde{\omega}$. The two ways to describe the reduction datum are essentially equivalent since the level function $\ell$ as well as the scaled reduction $\tilde{\omega}$ are determined by the graded reduction $\tilde{\omega}^\gr$. However, in this paper we work with actual differential forms rather than sections of a projectivized Hodge bundle, thus the language of graded reductions is slightly more natural in our setting. Also, it is convenient to have one object encoding simultaneously tropical and reduction datum similarly to metrized curve complexes.
\end{rem}


\subsection{The residue function}\label{sec:residuefuniction}
The last ingredient of the tropical reduction datum is the {\em residue function} $\Re_\omega\:E(\Gamma)\to k$. Although, $\Re_\omega$ takes values in the ground field $k$ rather than in the residue field $\tilde{k}$ it is very convenient to include it in the tropicalization datum as we explain below.

\subsubsection{The residues along an annulus}
Let $\CA=\CM(k\{t,rt^{-1}\})$ be an annulus with skeleton $e\cong[r,1]$. By an {\em orientation} on $\CA$ we mean a choice of an equivalence class of coordinates such that the two coordinates are equivalent if and only if either both are decreasing on $e$ or both are increasing on $e$. We define the {\em induced orientation} on the skeleton $e$ to be the one along which the coordinate is decreasing, i.e., the head of $e$ with respect to the induced orientation from the coordinate $t$ on $\CA$ is $\fh(e)=r$, and the tail is $\ft(e)=1$. Although this choice may seem strange, it simplifies various formulae involving residues of differential forms.

Let $\omega$ be a differential form without zeros and poles on an oriented analytic annulus $\CA=\CM(k\{t,rt^{-1}\})$. Then $\omega=(\sum_{-\infty}^\infty c_nt^n)\frac{dt}{t}$, and we define the {\em residue} of $\omega$ along $\CA$ to be $\Res_\CA(\omega):=c_0$. Note that it is invariant under the orientation-preserving changes of coordinate in any characteristic, see \cite{B19}. We use the capital letter in the notation, to distinguish between this type of residue and the usual one, which will be denoted by $\res_y(\omega)$.

For completeness, we give a simpler proof of the fact that $\Res_\CA(\omega)$ is well-defined that applies in our case. Notice that $c=c_0$ is the only scalar for which the form $\omega-c\frac{dt}{t}$ is exact. If $s=s(t)$ is another coordinate on $\CA$ compatible with the orientation then the dominant monomial of $u(t):=\frac{s}{t}$ is of degree zero, and hence the series for $\log(u)$ converges under our assumption that $k$ is of pure characteristic zero. Thus, $\frac{ds}{s}-\frac{dt}{t}=\frac{du}{u}$ is exact, and hence so is $\omega-c_0\frac{ds}{s}$.

It follows from the definition that the residue is {\em alternating}, i.e., changing the orientation on $\CA$ gives rise to the change of sign of the residue. Indeed, changing the orientation corresponds to the choice of coordinate $s=rt^{-1}$ on $\CA$, and the coefficient of $\frac{ds}{s}$ in $\omega$ is clearly $-c_0$.

\subsubsection{Residues at branches}
Assume now that $X$ is a nice curve equipped with a form $\omega$. It follows from the semistable reduction theorem, that any closed interval $e\subset X$ possesses a finite subdivision to intervals $e_i=[a_i,a_{i+1}]$ such that each interior $(a_i,a_{i+1})$ is the skeleton of an open (semi-)annulus $\CA_i\subset X$ containing neither zeros nor poles of $\omega$. Then $\Res_{\CA_i}(\omega)$ is a well defined element of $k$, as it is defined on closed subannuli in a compatible way. In particular, it follows that for any point $x$ of type 2 and a branch $e\in \Br(x)$ oriented away from $x$, the residue $\Res_e(\omega)$ is a well defined element of $k$. Here are simple basic properties of $\Res$.

\begin{lem}\label{branchlem}
Let $X$ be a nice curve with a meromorphic form $\omega$. Then,

(i) if $x$ is of type $1$ and $e\in\Br(x)$ then $\Res_e(\omega)=-\res_x(\omega)$;

(ii) if $x$ is of type $2$ and $e\in\Br(x)$ then $\widetilde{\Res_e(\omega)}^\gr=\res_{p_e}(\tilomega^\gr)$.
\end{lem}
\begin{proof}
To prove (i), consider the Laurent decomposition $\omega=(\sum_{i=n}^\infty  c_it^i)\frac{dt}{t}$ in a punched disc around $x$. Then, the orientation of $e$ is not compatible with the coordinate $t$, and hence $\res_x(\omega)=c_0=-\Res_e(\omega)$.

To prove (ii), pick $c\in k^\times$ with $|c|^{-1}=\|\omega\|_x$, and note that it suffices to check the claim for $c\omega$. Thus, we may assume that $\|\omega\|_x=1$ and $\tilomega^\gr$ is the usual reduction. Choose $t\in\CO^\circ_{X,x}$ such that $\tilt$ is a meromorphic function on $C_x$ with a simple zero at $p_e$. Then $t$ is a coordinate on a small enough open annulus $\CA$ along the branch of $p_e$ and hence $\omega=(\sum_n c_nt^n)\frac{dt}{t}$. Furthermore, $|c_i|_x\le 1$ and $(\sum_n \tilc_n\tilt^n)\frac{d\tilt}{\tilt}$ is the Laurent series of $\tilomega$ at $p_e$. In particular, $\res_{p_e}(\tilomega)=\tilc_0=\widetilde{\Res_e(\omega)}$.
\end{proof}

\subsubsection{The residue function}\label{resfunsec}
Assume now that $\Gamma$ is a skeleton of $(X,\omega)$. Then each edge $e$ is the skeleton of an open annulus or a punched disc $\CA$ and we set $\Re_\omega(e):=\Res_\CA(\omega)$, obtaining an alternating function $\Re_\omega\: E(\Gamma)\to k$ compatible with classical residues and reduction by Lemma~\ref{branchlem}:

\begin{prop}\label{prop:residuefunction}
Keep the above notation. Then,

(i) $\Re_\omega(e^\op)=-\Re_\omega(e)$ for any bounded edge;

(ii) $\Re_\omega(l)=\res_{q_l}(\omega)$ for any leg $l$ adjacent to a point $q_l$ of type $1$;

(iii) $\widetilde{\Re_\omega(e)}^\gr=\res_{p_e}(\tilomega^\gr_x)$ for any vertex $x\in\Gamma$ of type $2$ and any $e\in\Star(x)$.
\end{prop}

\begin{rem}
Notice that $|\Re_\omega(e)|\le \min\{\|\omega\|_{\ft(e)},\|\omega\|_{\fh(e)}\}$, and it is possible that the inequality is strict, which implies that no evidence of the residue can be seen at the components $\tilde{\omega}_{\ft(e)}$ and $\tilde{\omega}_{\fh(e)}$ of the scaled reduction of $\omega$.
\end{rem}

\section{The main results}
In this section we formulate our main results, proofs will be given in Section~\ref{proofs}.

\subsection{The harmonicity of the residue function}
Our first main result, whose proof relies on the study of differential forms on analytic annuli, asserts that the residue function $\Re_\omega$ is harmonic:

\begin{thm}\label{thm:harmonicity}
If $x\in V_2(\Gamma)$ is not a boundary point then $$\sum_{e\in\Br(x)}\Res_e(\omega)=\sum_{e\in\Star(x)}\Re_\omega(e)=0.$$
\end{thm}

If $x\in V(\Gamma)$ and $e\in\Br(x)$ is a branch not in $\Gamma$, then $e$ leads to a disc on which $\omega$ is regular, and hence $\Res_e(\omega)=0$. This explains why the infinite sum on the right-hand side is well-defined, and implies the first equality. The proof of the second equality will be given in \S\ref{proofharmonic}.

\begin{rem}
The alternating and harmonicity properties of the residue function have a series of ``traces'' over the residue field. Indeed, if $G\subset \Gamma$ is a full subgraph all of whose vertices are of type $2$ and do not belong to the boundary $\partial X$ then
$$\sum_{x\in V(G), e\in Star(x)\cap E(G)}\Re_\omega(e)=0$$
by the alternating property. On the other hand, by the harmonicity of $\Re_\omega$, we have
$$\sum_{x\in V(G), e\in Star(x)}\Re_\omega(e)=0,$$
and hence
\begin{equation}\label{eq:GRCk}
\sum_{x\in V(G), e\in Star(x)\cap L(\Gamma)}\Re_\omega(e)=\sum_{x\in V(G), e\in Star(x)\setminus (E(G)\cup L(\Gamma))}\Re_\omega(e^\op).
\end{equation}
In particular, if $X$ has no boundary, $G$ is a connected component of the full subgraph of all vertices of type $2$ of level greater than $\ell_0$, and along any leg $l$ with $\ft(l)\in G$, the level function $\ell$ has negative slope then the left-hand side of \eqref{eq:GRCk} vanishes, and hence so does the right-hand side. Set $r:=10^{-\ell_0}$. Then the vanishing of the $r$-graded component of the graded reduction of the right-hand side is the {\em global residue condition} of \cite{BCGGM}.
\end{rem}

To formulate other results we will need the following definition.

\subsubsection{The tropical reduction datum}
Let $\gamma=(\mathcal{C}, \tilde{\eta}^\gr, \Re)$ be a triple consisting of a metrized curve complex with boundary $\mathcal{C}=\left(\Gamma,(C_x)_{x\in V_2(\Gamma)},(p_e)_{e\in E(\Gamma)}\right)$, a collection of graded elements $\tilde{\eta}^\gr_x\in \Omega_{\tilde{k}(C_x)/\tilde{k}}\otimes_{\tilde{k}}k^\gr$ for $x\in V_2(\Gamma)$, and a function $\Re\:E(\Gamma)\to k$. Let $\ell_\gamma\:\Gamma\to \RR$ be the unique continuous, piecewise affine function such that (i) $\ell_\gamma$ is affine on the edges of $\Gamma$, (ii) $10^{\ell_\gamma(x)}$ is the grading of $\tilde{\eta}^\gr_x$ for all $x\in V_2(\Gamma)$, and (iii) the slope of $\ell_\gamma$ along any leg $l$ is equal to the minus log-order of $\tilde{\eta}^\gr_x$ at $p_l$, where $x=\ft(l)$ is the vertex of type $2$ adjacent to $l$.

\begin{defin}
A triple $\gamma$ is called a {\em tropical reduction datum} if the following compatibilities hold:
\begin{enumerate}
\item for any $e\in E(\Gamma)$ with tail $x$, the log-order of $\tilde{\eta}^\gr_x$ at $p_e$ is $-\frac{\partial\ell_\gamma}{\partial e}$;
\item for any $e\in E(\Gamma)$ with tail $x$, we have $\widetilde{\Re(e)}^\gr=\res_{p_e}(\tileta_x^\gr)$;
\item if $x\in V_2(\Gamma)$ is not boundary then $\sum_{e\in\Star(x)}\Re(e)=0$;
\item if $e\in L(\Gamma)$ if $\frac{\partial\ell_\gamma}{\partial e}<0$ then $\Re(e)=0$.
\end{enumerate}
\end{defin}
Notice that if $\gamma$ is a tropical reduction datum and $e\in\Star(x)$ is an edge, then the slope of $\ell_\gamma$ along $e$ is negative if and only if $\tilde{\eta}^\gr_x$ is regular at $p_e$. Furthermore, if $\frac{\partial\ell_\gamma}{\partial e}=0$ then $\tilde{\eta}^\gr_x$ has a simple pole at $p_e$ by condition (1). Thus, $\res_{p_e}(\tileta_x^\gr)\ne 0$, and hence $\Re(e)\ne 0$ by condition (2). In particular, $\ell_\gamma(x)=-\val(\Re(e))$.

\subsection{Tropical reduction}
Combining Propositions~\ref{prop:gradedred} and \ref{prop:residuefunction}, and Theorem~\ref{thm:harmonicity}, we obtain the following reduction result.

\begin{thm}\label{thm:tropicalization}
Let $X$ be a nice $k$-analytic curve equipped with a non-zero meromorphic differential form $\omega$. Let $\gamma(X,\Gamma,\omega)$ be the triple consisting of the metrized curve complex with boundary $\CC$ associated to $(X,\Gamma)$, the graded reduction $\tilde{\omega}^\gr$, and the residue function $\Re_\omega$. Then $\gamma(X,\Gamma,\omega)$ is a tropical reduction datum.
\end{thm}

\subsection{Lifting}
The main result of the paper is the following lifting theorem.

\begin{thm}\label{thm:lifting}
For any tropical reduction datum $\gamma$, there exists a nice $k$-analytic curve $X$ with a skeleton $\Gamma$ and a non-zero meromorphic differential form $\omega$ such that $\gamma(X,\Gamma,\omega)=\gamma$. 
\end{thm}

\subsection{Relation to \cite{BCGGM}}
Finally, let us explain the relation to other lifting results. We will use the notation from \S\ref{compactsec}. Let $\tilk$ be an algebraically closed field of residual characteristic zero and $k=\widehat{\cup_n\tilk((t^{1/n}))}$.

\subsubsection{The crude lifting theorem}\label{crude}
\cite[Theorem~1.3]{BCGGM} asserts that a $\tilk$-point $\gamma_0=(\tilX,p,\tilomega)$ of $\PP\Omega\oCM_{g,n}$ belongs to $K\CM_{g,\mu}$ if and only if it lifts to a tropical reduction datum $\gamma$ over $\tilk$. In {\em loc.cit.}, this was deduced from \cite[Proposition~4.8]{BCGGM}, and it follows from our results in a similar way. Indeed, if $\gamma_0\in K\CM_{g,\mu}$, then there exists a morphism $\Spec(\kcirc)\to K\CM_{g,\mu}$ which takes the closed point to $\gamma_0$ and the generic point to $(X,q,\omega)$ lying in $\PP\Omega\CM_{g,n}$. Applying Theorem~\ref{thm:tropicalization} to the pullback of $(X,q,\omega)$ to $k$ we obtain a tropical reduction datum $\gamma$ which lifts $\gamma_0$. Conversely, if $\gamma_0$ lifts to a tropical reduction datum $\gamma$ then we can lift $\gamma$ to a tuple $(X,q,\omega)$ over $k$ by Theorem~\ref{thm:lifting}. The induced morphism $\Spec(k)\to\PP\Omega\CM_{g,n}$ extends to a morphism $\Spec(\kcirc)\to K\CM_{g,\mu}$ by the properness of the target and the valuative criterion. Thus, $\gamma_0$ is the image of the closed point in $K\CM_{g,\mu}$.

\subsubsection{The fine lifting theorems}
Note that \cite[Proposition 4.8]{BCGGM} is a precise analogue of Theorem~\ref{thm:lifting}, but we construct a formal lifting while \cite{BCGGM} constructs a lifting to a complex analytic family. It seems that neither result implies another one in a simple way, though our proof can perhaps be adapted to produce a complex analytic family too by lifting $\gamma$ to the subfield of convergent power series in $\cup_n\CCC((t^{1/n}))$. This mainly reduces to the check that the process for finding good coordinates in Proposition~\ref{goodprop} below preserves complex analytic convergence of the coefficients, but we will not pursue this direction.

Note also that equivalence of the two fine lifting theorems follows from the very difficult main result of \cite{BCGGM2} which constructs a finer proper moduli space whose points are described by tropical reduction data. This is done by the same argument as used in \S\ref{crude}.

\section{Proofs}\label{proofs}
\subsection{Good coordinates on analytic annuli}\label{sec:goodcoord}
Let $\CA$ be an oriented annulus of modulus $0<r<1$ and $\omega$ a differential form having neither zeros nor poles on $\CA$.

\begin{defin}
An analytic coordinate $t$ on $\CA$ is called {\em good} with respect to $\omega$ if either $\omega= c_0\frac{dt}{t}$ or $\omega=(c_nt^n+c_0)\frac{dt}{t}$, $n\ne 0$, and $|c_nt^n|_x>|c_0|_x$ for all $x\in\CA$.
\end{defin}

\begin{rem}
If $\omega=(c_nt^n+c_0)\frac{dt}{t}$ with $n\ne 0$ then $n$ is nothing but minus the slope of the level function $\ell(x):=\log\|\omega\|_x$ on the skeleton of $\CA$. Similarly, if $\omega= c_0\frac{dt}{t}$ then the level function is constant on the skeleton of $\CA$.
\end{rem}

\begin{prop}\label{goodprop}
If a form $\omega$ has neither zeros nor poles on an oriented analytic (semi-)annulus $\CA$, then $\CA$ admits a good coordinate with respect to $\omega$.
\end{prop}

\begin{proof}
We will assume that $\CA$ is closed. Since the construction we present is canonical, in the case of an open (semi-)annulus it produces a compatible set of good coordinates on all closed subannuli, and hence also on $\CA$. Pick any coordinate $s$ on $\CA$ compatible with the orientation. Then $\omega=\sum a_is^i\frac{ds}{s}$. Since $\omega$ has neither zeros nor poles on $\CA$, there exists $n\in\ZZ$ such that
\begin{equation}\label{eq:ineqm1}
\left|a_ns^n\right|_x>\left|a_is^i\right|_x
\end{equation}
for all $x\in\CA$ and all $i\ne n$.

The case $n=0$ is easy, and we explain it first. We are looking for a unit $u(s)$ such that the coordinate $t=su(s)$ is good, i.e., $u$ is a solution of the differential equation
$\sum a_is^i\frac{ds}{s}=\omega=a_0\frac{dt}{t}=a_0\frac{ds}{s}+a_0\frac{du}{u},$
or equivalently, $\frac{du}{u}=\sum_{i\ne 0}\frac{a_i}{a_0}s^i\frac{ds}{s}$. The later can be solved explicitly by setting $u(s):=\exp(\sum_{i\ne 0}\frac{a_i}{ia_0}s^i)$, which is a unit converging on $\CA$ by \eqref{eq:ineqm1}.

Assume now that $n\ne 0$. Unlike the case $n=0$, we cannot write the coordinate $t$ explicitly. Instead, we construct it by a converging sequence of approximations. For a coordinate $t$, set
$$\epsilon_\omega(t):=\frac{\|\omega-(c_nt^n+c_0)\frac{dt}{t}\|_\CA}{\|c_nt^n\|_\CA}<1,$$ where $\omega=\sum c_it^i\frac{dt}{t}$. It is sufficient to construct a coordinate $t$ for which $\epsilon_\omega(t)=0$.

As a first approximation, we construct a coordinate $s_1$ for which $$\omega-a_0\frac{ds}{s}=a_ns_1^n\frac{ds_1}{s_1}.$$ It is given by $s_1=su_1$, where $u_1=\sqrt[n]{\sum_{i\ne 0}\frac{n}{i}\frac{a_is^i}{a_ns^n}}$ is unit on $\CA$ by \eqref{eq:ineqm1}. It follows that
$$\omega=a_ns_1^n\frac{ds_1}{s_1}+a_0\frac{ds}{s}=(a_ns_1^n+a_0)\frac{ds_1}{s_1}-a_0\frac{du_1}{u_1}=\sum a'_i s_1^i\frac{ds_1}{s_1}.$$
By construction, $\|u_1-1\|\le \epsilon_\omega(s)<1$ and $\|du_1\|\le \epsilon_\omega(s)<1$. Thus, $|a_n'|=|a_n|$ and hence
$$\epsilon_\omega(s_1)\le\frac{\left\|a_0\frac{du_1}{u_1}\right\|_\CA}{\|a'_ns_1^n\|_\CA}\le\epsilon_\omega(s)\frac{|a_0|}{\|a_ns^n\|_\CA}\le(\epsilon_\omega(s))^2.$$

Repeating the construction above, we produce a sequence of coordinates $s_j$ on $\CA$ and a sequence of units $u_j$ satisfying the following:
$$
\begin{array}{c}
    s_j=s\prod_{i=1}^j u_i, \\
    \epsilon_\omega(s_j)\le (\epsilon_\omega(s))^{j+1}, \\
    \|u_j-1\|\le (\epsilon_\omega(s))^j.
\end{array}
$$
Since $\epsilon_\omega(s)<1$, it follows that the products $\prod_{i=1}^j u_i$ converge to a unit $u$, the coordinates $s_j$ converge to a coordinate $t=su$, and $\epsilon_\omega(t)=\lim \epsilon_\omega(s_j)=0$ as needed.
\end{proof}

\subsection{Proof of Theorem~\ref{thm:harmonicity}}\label{proofharmonic}
We reduce the assertion to the classical statement about algebraic curves.
Let $e\in\Star(x)$ be an edge, and $\CA$ an annulus whose skeleton is supported on $e$ and has compatible orientation. By Proposition~\ref{goodprop}, there exists a good analytic coordinate $t$ on $\CA$, i.e., $\omega|_{\CA}=(c_nt^n+c_0)\frac{dt}{t}$, where $c_0=\Res_\CA(\omega)$. Re-scaling $t$ we can assume that $\CA=\CM(k\{t,rt^{-1}\})$, where $r$ is the modulus of $\CA$. Let us truncate the curve $X$ at the end point of $\CA$ and glue in an open unit disc $D$ along $\CA$ by identifying the coordinate $s$ of the disc with $t$. By construction, $\omega|_{\CA}$ extends to $D$ by the formula $\omega|_{D}=(c_ns^n+c_0)\frac{ds}{s}$, and this form has residue $c_0$ at the origin and is regular outside the origin. Repeating this process for all edges in $\Star(e)$ we obtain a complete analytic curve $Y$ with a form $\omega_Y$ having total residue $\sum_{e\in\Star(x)}\Re_\omega(e)$. By GAGA, $(Y,\omega_Y)$ is the analytification of an algebraic curve with differential. And hence its total residue vanishes.

\subsection{Proof of Theorem~\ref{thm:lifting}}
The general plan is as follows: cut $\gamma$ into stars, lift the datum to star-shaped curves, and then glue the local lifts.

\subsubsection{Lifting the form}\label{liftform}
We start with lifting a local datum, and ignore residues at first. In this case we can just use a constant family over $\kcirc$ and restrict to the generic fiber.

\begin{lem}\label{liftlem}
Assume that $C$ is a smooth proper curve over $\tilk$ with a finite set $p=\{p_1\.p_n\}$ of closed points and a meromorphic differential form $\tileta$ such that $\divisor(\tileta)=\sum_i m_ip_i$ for $m_i\in\ZZ$. Then there exists a nice proper $k$-curve $Y$ with reduction $C$, a set of $k$-points $q=(q_1\.q_n)$ lifting $p$, and a lift $\eta$ of $\tileta$ such that $\divisor(\eta)=\sum_i m_iq_i$.
\end{lem}
\begin{proof}
By the theory of coefficient fields, the homomorphism $\kcirc\to\tilk$ admits a section $\tilk\into\kcirc$. This yields a homomorphism $\tilk\into k$, and we simply take our lifts $Y$, $q$ and $\eta$ to be the analytifications of $C\otimes_\tilk k$, $p\otimes_\tilk k$ and $\tileta\otimes_\tilk k$.
\end{proof}

\begin{rem}\label{liftrem}
One could wonder if for a given lift $Y$ of $C$, there exists a lift $\eta$ of the differential form preserving the profile of zeros and poles. By \cite[Theorem~1]{KZ03}, there exist components of the moduli space of stratified differentials supported over the hyperelliptic locus of the moduli space of curves. Thus, it may happen that $\tileta$ is liftable only for particular choices of the lift $Y$. Notice that if $p_1\.p_r$ are the poles of $\tileta$ then for any lift $(Y,q_1\.q_r)$ of $(C,p_1\.p_r)$ one can lift $\tileta$ so that $\sum_{i=1}^rm_iq_i$ is the polar part of $\divisor(\eta)$. Indeed, if $r>0$ then the lift exists because the obstruction $H^1\left(C,\Omega^1_{C/\tilde{k}}(-\sum_{i=1}^r m_ip_i)\right)$ vanishes by Serre's duality. If $r=0$, then consider the smooth model $f\:\CY\to\Spec(\kcirc)$ with the closed fiber $C$, and note that $\tileta$ lifts to a section of $\Omega^1_{\CY/\kcirc}$ because $f_*(\Omega^1_{\CY/\kcirc})$ is a vector bundle. However, such lift provides no control on the zeros of $\eta$, and, in particular, the zeros of high multiplicity can split.
\end{rem}

\subsubsection{Lifting the residues}\label{liftres}
In order to correct the residues locally at $y$ we will use the following result.

\begin{lem}\label{reslem}
Assume that $Y$ is a proper nice $k$-curve with a good reduction $C_y$, $y$ is the point of type $2$ mapped to the generic point of $C_y$ by the reduction map, $q_1\.q_m\in Y$ points of type $1$ with distinct reductions $p_1\.p_m\in C_y$, and $a_1\.a_m\in k$ elements summing up to 0. If not all $a_i$ vanish then there exists $\eta\in H^0\left(Y,\Omega^1_{Y/k}(\sum_{i} q_i)\right)$ such that $\|\eta\|_y=\max_i|a_i|$ and $\res_{q_i}(\eta)=a_i$ for all $i$.
\end{lem}
\begin{proof}
Let $W$ be the set of differential forms $\eta\in H^0\left(Y,\Omega^1_{Y/k}(\sum_{i} q_i)\right)$ for which $\res_{q_i}(\eta)=a_i$ for all $i$. Recall that by Serre's duality the sequence
$$0\to H^0(Y,\Omega^1_{Y/k})\to H^0\left(Y,\Omega^1_{Y/k}\left(\sum q_i\right)\right)\to \sum_{i=1}^mk\xrightarrow{\sum} k\to 0$$
is exact. Thus, $W$ is not empty. Furthermore, $W=\eta_0+H^0(Y,\Omega^1_{Y/k})$, where $\eta_0\in W$ is any form. Set $l_i:=[y,q_i]\subset Y$. Then,  by Proposition~\ref{prop:residuefunction}, $\Re_\eta(l_i)=\res_{q_i}(\eta)=a_i$ for all $i$ and $\eta\in W$. In particular, it follows that $\|\eta\|_y\ge\max_i|a_i|$, and our next goal is to minimize $\|\eta\|_y$.

Recall that by \cite[Definition~2.4.3/2]{BGR}, a normed $k$-vector space $V$ is {\em $k$-Cartesian} if and only if any finite dimensional subspace $U\subseteq V$ is {\em strictly closed}, i.e., for any $v\in V$ there exists a vector $u\in U$ closest to $v$. Furthermore, $k(Y)$ with the norm $|\ |_y$ is $k$-Cartesian by \cite[Proposition~5.3.3/2]{BGR}, and multiplication by a meromorphic form $\omega$ with $\|\omega\|_y=1$ induces an isometry between this space and the space $\Omega^1_{k(Y)/k}$ with the norm $\|\ \|_y$. Thus, there exists a form $\eta_1\in H^0(Y,\Omega^1_{Y/k})$ such that $\|\eta_0-\eta_1\|_y$ is minimal possible. Set $\eta:=\eta_0-\eta_1$.

We claim that the graded reduction $\tileta^\gr$ is not regular. Otherwise, by the argument from Remark~\ref{liftrem} it can be lifted to a regular form $\eta'\in H^0(Y,\Omega^1_{Y/k})$ and then $\|\eta-\eta'\|_y<\|\eta\|_y$, which contradicts the minimality of $\|\eta\|_y$. Let $p$ be a pole of $\tileta^\gr$, then its preimage $D_p\subset Y$ contains a pole $q_i$ of $\eta$. Since $q_i$ is a simple pole and no other poles lie in $D_p$, it follows that $p$ is a simple pole too. The graded reduction of $a_i=\Re_\eta(l_i)$ equals $\res_p(\tileta)$ by Proposition~\ref{prop:residuefunction}, and the latter does not vanish since the pole is simple. Thus, $\|\eta\|_y=|a_i|$, which concludes the proof.
\end{proof}

\subsubsection{The construction of local lifts}
Let $y\in\Gamma$ be a vertex of type $2$, and $p_1\.p_n\in C_y$ the points corresponding to $\Star(y)$. Let $\oC_y\setminus C_y=\{p_{n+1}\.p_m\}$, where $\oC_y$ is the smooth compactification of $C_y$. By Lemma~\ref{liftlem}, there exists a smooth proper $k$-curve $\oY$ with reduction $\oC_y$, a collection of $k$-points $q=(q_1\.q_m)$ that lifts $p$, and a form $\eta$ that lifts the scaled reduction $\tileta_y$ such that the support of  $\divisor(\eta)$ belongs to $q$.

For $i=1\. m$, set $l_i:=[y,q_i]$ and $a_i:=\Re_\eta(l_i)$, where $y\in\oY$ is the preimage of the generic point of $C_y$. After multiplying $\eta$ by an appropriate $c\in k^\times$, we may assume that the graded reduction of $\eta$ is $\tileta_y^\gr$. It follows that for $i\le n$, the graded reductions of $a_i$ and $\Re(l_i)$ are equal to $\res_{p_i}(\tileta_y^\gr)$, and hence $a'_i=\Re(l_i)-a_i$ satisfies $|a'_i|<\|\eta\|_y$. In the non-boundary case, $m=n$ and the $a'_i$'s sum up to zero. Otherwise, set $a'_{n+1}:=-\sum_{i=1}^na'_i$ and $a'_i:=0$ for $i\ge n+2$. Thus, in any case, the $a'_i$'s sum up to zero.

If $a'_i=0$ for all $i$ then set $\omega_y:=\eta$. Otherwise, by Lemma~\ref{reslem}, there exists a form $\eta'$ on $\oY$ with at worst simple poles in $q$ such that $\Re_{\eta'}(l_i)=a'_i$ for all $i$ and $\|\eta'\|_y<\|\eta\|_y$. Set $\omega_y:=\eta+\eta'$. We constructed a lift $\omega_y$ of $\tileta_y^\gr$ to $\oY$ with desired residues around $y$, namely $\Re_{\omega_y}(e_i)=a_i$ for all $i$, where $e_i\in \Br(y)$ denotes the branch of $l_i$. Note however, that $\omega_y$ may have zeros outside of $q$.

We consider all open discs of $\oY\setminus\{y\}$ as having radius one. Let $Y$ be the nice curve obtained from $\oY$ by removing the open discs of radius one around $q_{n+1}\.q_m$, and removing large enough open discs around $q_1\.q_n$ of radius smaller than 1 so that the slope of $\|\omega_y\|$ along the edge $e_i\subset l_i\cap Y$ is constant, and the length of $e_i$ is smaller than one half of the length of the corresponding edge in $\Star(y)$. Then $(Y,y)$ is a star-shaped curve with a star skeleton $\Gamma_y$ whose edges are $e_1\.e_n$, in particular, $\Gamma_y$ is obtained from $\Star(y)$ by shortening the edges. Furthermore, $\omega_y$ is a differential form on $Y$ such that $\gamma(Y,\Gamma_y,\omega_y)$ is the restriction of $\gamma$ onto the subgraph $\Gamma_y$. Strictly speaking, one also inserts $\PP^1$'s in the ends of $e_i$, etc., but these details are not essential, since we can always cut the edges further.

\subsubsection{The patching of local lifts}
In the last step, we show how to patch the star-shaped nice $k$-analytic curves constructed in the previous step in order to get a nice $k$-analytic curve representing $\gamma$. There are two types of patches we should do.

First, assume that $l=[y,q]$ is a leg. It corresponds to an annulus $\CA$ in $Y\setminus\{y\}$ with a form $\omega_y$, and by Proposition~\ref{goodprop}, there exists a good analytic coordinate $t$ on $\CA$ such that $\omega_y|_{\CA}=(c_nt^n+\Re(e))\frac{dt}{t}$, where $n$ is minus the slope of the level function along the oriented edge $e$, and $c_n=0$ if $n=0$. In addition, if $\Re(e)\neq 0$, then $n\le 0$. After rescaling $t$, we may assume that $\CA=\CM(k\{t,rt^{-1}\})$. Consider an open unit disc $D_q$ with a coordinate $s$ and a form $\omega_q=(c_ns^n+\Re(e))\frac{ds}{s}$ and glue $Y$ and $D_q$ along $\CA$ via $t=s$. Clearly, $\omega_y$ and $\omega_q$ agree on $\CA$.

Next, assume that $e\in E(\Gamma)$ is bounded, and set $y:=\ft(e)$ and $z:=\fh(e)$. The cases $y=z$ and $y\neq z$ are similar, so assume that $y\neq z$. Consider the star-shaped local lifts $(Y,\omega_y)$ and $(Z,\omega_z)$ with open annuli $\CA_Y\subset Y$ and $\CA_Z\subset Z$. Then the orientation of $\CA_Y$ is compatible with $e$, and of $\CA_Z$ is not. By our assumption, each of them has modulus larger than $\sqrt{r}$, where $r=10^{-|e|}$. By Proposition~\ref{goodprop}, there exist good analytic coordinates $t$ and $s$ on the annuli such that $\omega_Y|_{\CA_Y}=(\alpha t^n+\Re(e))\frac{dt}{t}$ and $\omega_Z|_{\CA_Z}=(\beta s^{-n}+\Re(e^\op))\frac{ds}{s}$, where $n$ is minus the slope of the level function along $e$, and $\alpha=\beta=0$ if $n=0$. Choose an open annulus $\CA$ of modulus $r$ with coordinate $\tau$ and a form $\omega_\CA=(\alpha \tau^n+\Re(e))\frac{d\tau}{\tau}$ and glue it to $\CA_Y$ and $\CA_Z$ via $\tau=t$ and $\tau=s^{-1}(-\beta/\alpha)^{1/n}$.

Once all these gluings are done, we obtain a nice curve $X$ with a skeleton $\Gamma\subset X$ such that the associated metrized curve complex is $\CC$. In addition, the local forms glue to a regular form $\omega$ on $X$, such that $\ell(\omega)=\ell_\gamma$, the graded reduction at each vertex $y\in V(\Gamma)$ is $\tileta_y^\gr$, and the residue along each edge $e\in E(\Gamma)$ is $\Re(e)$. Thus, $\gamma(X,\Gamma,\omega)=\gamma$, as required.

\section{The stacky tropical reduction}\label{sec5}
Although in the proof of the main results we did not need it, there is a natural stacky reduction datum that one may wish to take into account. In a sense, the situation is similar to the proof of the correspondence theorem in \cite{T12}. Starting with a regular tropical reduction of a curve with a map to a toric variety, one can prove a lifting result, but in order to have a one-to-one correspondence between tropical reductions and algebraic objects satisfying appropriate constraints, one has to consider richer reductions: either stacky reductions or the reductions in the category of log schemes.

\subsection{Good formal coordinates}
Let us describe the type of a stacky structure that appears naturally in the setting of curves with differentials. Let $C$ be a smooth curve over the residue field $\tilde{k}$, $\omega_C$ a non-zero meromorphic differential form on $C$, and $p\in C$ a closed point. Set $n:=\ord_p(\omega_C)+1\in\ZZ$ to be the logarithmic order of $\omega_C$ at $p$. We say that a formal coordinate $t\in \widehat{\CO}_{C,p}$ is {\em good} if $\omega_C=(c_nt^n+\res_p(\omega_C))\frac{dt}{t}$ for some $c_n\in k$. Similarly to Proposition~\ref{goodprop}, one can prove that good formal coordinates exist. In the complex case they are even holomorphic, see \cite[Thorem~4.1]{BCGGM2}. For any $n$, set $G_n:=\GG_m\ltimes_n \GG_a$, where $\GG_m$ acts on $\GG_a$ by $\lambda(\mu):=\lambda^{-n}\mu$.
\begin{prop}\label{prop:goodfortor}
The set of good formal coordinates admits a natural transitive action of the group $G_n(\tilk)$. If $n<0$ then the action is free. Otherwise, it factors through the free action of $\GG_m(\tilk)=G_n(\tilk)/\GG_a(\tilk)$.
\end{prop}
\begin{proof}
The case $n\ge 0$ is easy. Indeed, in this case $G_n(\tilk)$ acts on the set of formal coordinates via its quotient $\GG_m(\tilk)$ by the multiplication. Let $t,s$ be two good formal coordinates. If $n=0$ then $\frac{dt}{t}=\frac{ds}{s}$, which implies $t=cs$ for some $c\in\tilk^\times$. If $n>0$ then $\Res_p(\omega_C)=0$, and hence $t^{n-1}dt=cs^{n-1}ds$ for some $c\in\tilk^\times$. Thus, $t^n=cs^n$, and again, $t$ belongs to the $\tilk^\times$-orbit of $s$.

Assume now that $n=-l<0$, and set $r:=\res_p(\omega_C)$. Pick a good formal coordinate $s$. Then $\omega_C=(es^n+r)\frac{ds}{s}$ for some $e\ne 0$. In order to explain how the group $G_n(\tilk)$ acts on $s$ we shall analyze the set of good formal coordinates. Let $t=\sum_{i\ge 1}a_is^i$, $a_1\ne 0$, be the expansion of a good formal coordinate $t$. Set $u:=\frac{t}{s}$. There exists $c\in\tilk^\times$ such that $\omega_C=(ct^n+r)\frac{dt}{t}=(es^n+r)\frac{ds}{s}$, and hence
\begin{equation}
(c+rt^{l})\frac{dt}{ds}=u^{l+1}(e+rs^{l}).
\end{equation}

Let us expand this equation, and for all $k\ge 1$, compare the first contribution of $a_k$, which happens on the level of the coefficients of $s^{k-1}$:

\begin{equation}\label{eq:torgc}
(c+r(a_1s+a_2s^2+\dots)^l)(a_1+2a_2s+\dots)=\\(a_1+a_2s+\dots)^{l+1}(e+rs^{l}).
\end{equation}

If $k=1$ we get $ca_1=ea_1^{l+1}$, or equivalently, $ea_1^{l}=c$. If $2\le k\le l$ we get
\begin{equation}\label{eq:coefcomp1}
cka_k=e(l+1)a_1^{l}a_k+f_k=(l+1)ca_k+f_k,
\end{equation}
where $f_k$ belongs to the ideal generated by $a_2,\dotsc,a_{k-1}$ in $\tilk[a_1,a_2,\dotsc]$. In particular, for $k=2$, \eqref{eq:coefcomp1} is equivalent to $2ca_2=(l+1)ca_2$, and since $c\ne 0$, we get $a_2=0$. Similarly, $a_k=0$ for all $2\le k\le l$. For $k=l+1$, we get
$$cka_k+ra_1^k=e(l+1)a_1^{l}a_{l+1}+f_k+ra_1^{l+1}=(l+1)ca_{l+1}+f_k+ra_1^{l+1},$$
and since $a_2=\dots=a_{l}=0$, the latter equation is trivial.

Finally, for any $k>l+1$, comparing the coefficients of $s^{k-1}$ in \eqref{eq:torgc}  gives rise to an equation
$$cka_k+g_k=e(l+1)a_1^{l}a_k+h_k=(l+1)ca_k+h_k,$$
where $g_k$ and $h_k$ are explicit polynomials in $a_1,\dotsc a_{k-1}$. Since $k\ne l+1$ and $c\ne 0$, the coefficient $a_k$ is uniquely determined by $a_1$ and $a_{l+1}$ for all $k$.

We are ready to describe the action of $G_n(\tilk)$ on the set of good formal coordinates. An element $\sigma=(\lambda,\mu)\in (\tilk^\times)\ltimes_n\tilk$ associates to the good formal coordinate $s$ the unique good formal coordinate $\sigma(s)=\sum_{i\ge 1}a_is^i$ for which $(a_1,a_{l+1})=(\lambda,\lambda\mu)$. It remains to show that the formula above respects the group law on $G_n(\tilk)$. Pick $\sigma'=(\lambda',\mu')\in (\tilk^\times)\ltimes_n\tilk$, and let $a'_i$ be such that $\sigma'(s)=\sum_{i\ge 1}a'_is^i$. Then, $(a'_1,a'_{l+1})=(\lambda',\lambda'\mu')$, and $a_k=a'_k=0$ for all $2\le k\le l$. Therefore,
$$\sigma'(\sigma(s))=\sum_{i\ge 1}a'_i(\sigma(s))^i\equiv a'_1a_1s+(a'_1a_{l+1}+a'_{l+1}a_1^{l+1})s^{l+1}\mod s^{l+2}.$$
By definition, $\sigma\sigma'=(\lambda\lambda',\mu+\lambda^{l}\mu')$, and hence
$$(\sigma\sigma')(s)\equiv \lambda\lambda' s+(\lambda'\lambda\mu+\lambda^{l+1}\lambda'\mu')s^{l+1}\equiv \sigma'(\sigma(s))\mod s^{l+2}.$$
However, a good formal coordinate is determined by its $(l+1)$-jet in $s$. Therefore, $\sigma'(\sigma(s))=(\sigma\sigma')(s)$, which completes the proof.
\end{proof}
In the case of good coordinates on analytic annuli we have the following:
\begin{prop}
Let $\CA=\{x\in\CM(k\{s,\rho s^{-1}\})\,|\,\rho^{-1}<|s|_x<1\}$ be an open annulus, and $\omega$ a differential form on $\CA$ having neither zeros nor poles. Let $n$ be the slope of the level function associated to $\omega$. Then the set of good coordinates on $\CA$ compatible with the orientation admits a natural transitive action on $G_n(\ko)$, which is free if and only if $n\ne 0$. If $n=0$, the action factors through the natural free action of the quotient $\GG_m(\ko)$.
\end{prop}
\begin{proof}
If $n=0$, and $t,s$ are good formal coordinates compatible with the orientation then there exists a unit $u=u(s)$ such that $t=su(s)$. Since $\frac{dt}{t}=\frac{ds}{s}$, it follows that $\frac{du}{u}=0$, and hence $u\in (\ko)^\times$. Thus, the natural action of $\GG_m(\ko)$ on the set of good formal coordinates compatible with the orientation is free and transitive as asserted. Assume now that $n\ne 0$. Since for any $c\in k$ such that $|c|=\rho$ the map $s\mapsto cs^{-1}$ induces a bijection between the sets of good formal coordinates compatible with the orientation and not compatible with the orientation, it is sufficient to consider the case $n<0$. Set $l:=-n$.

Let $s,t=\sum_{i\in\ZZ}a_is^i$ be good formal coordinates compatible with the orientation. Then $|a_1|=1$, and $|a_ks^k|_x<|a_1s|_x$ for all $x\in\CA$. In particular, $|a_k|\le|a_1|$ for all $k$. It is sufficient to show that $a_i=0$ for all $i\le 0$. Indeed, in this case, the formal computation we did in the proof of Proposition~\ref{prop:goodfortor} shows that $\sum_{i\in\ZZ}a_is^i$ belongs to the $G_n(k)$ orbit of $s$. Furthermore, since $|a_1|=1$ and $|a_{l+1}|\le 1$, it follows that $\sum_{i\in\ZZ}a_is^i$ belongs to the $G_n(\ko)$ orbit of $s$. Vice versa, for any $\sigma\in G_n(\ko)$, let $\sigma(s)$ be the formal power series as in the proof of Proposition~\ref{prop:goodfortor}. It is easy to verify that all the coefficients of $\sigma(s)$ have absolute value at most $1$, and the coefficient of $s^1$ has absolute value $1$. Thus, $\sigma(s)$ converges on $\CA$ and has unique dominating monomial of degree $1$, i.e., it is a good coordinate on $\CA$ compatible with the orientation.

Let us show that $a_i=0$ for all $i\le 0$. Assume to the contrary that there exists $i\le 0$ such that $a_i\ne 0$. After shrinking $\CA$ if needed we may assume that there exists $p\le 0$ such that $|a_p|>|a_k|$ for all $p\ne k\le 0$. Since $t$ and $s$ are good, there exist $c\ne 0,e\ne 0,$ and $r$ in $k$ such that $\omega=(ct^n+r)\frac{dt}{t}=(es^n+r)\frac{ds}{s}$, and $|r|\le |c|=|e|$. After shrinking $\CA$ further we may assume that $|r|<|c|=|e|$. Finally, after replacing $t,s,$ and $\omega$ with appropriate multiples, we may assume that $c=e=1$. Thus,
$(1+rt^l)\frac{dt}{ds}=u^{l+1}(1+rs^{l})$, where $u(s)=\frac{t}{s}=\sum_{i\in\ZZ}ia_is^{i-1}$. Equivalently,
\begin{equation}\label{eq:goodcoordeq}
\frac{dt}{ds}-u^{l+1}=rs^lu^l\left[u-\frac{dt}{ds}\right].
\end{equation}
Comparing the absolute value of the free coefficients in \eqref{eq:goodcoordeq}, we conclude that $|a_1-a_1^{l+1}|$, and hence
\begin{equation}\label{eq:a1ismall}
|a_1^l-1|<1.
\end{equation}

Let us now compare the coefficient of $s^{p-1}$ in \eqref{eq:goodcoordeq}. Since $p-1<0$, $|a_1|=1\ge |a_k|$ for all $k$, and $|a_p|>|a_k|$ for all $p\ne k\le 0$, it follows that the absolute value of the coefficient of $s^{p-1}$ on the left hand side of \eqref{eq:goodcoordeq} is
$|pa_p-(l+1)a_1^la_p+\epsilon|,$
where $|\epsilon|<|a_p|$. Thus, by \eqref{eq:a1ismall}, the absolute value of the coefficient of $s^{p-1}$ on the left hand side of \eqref{eq:goodcoordeq} is
$$|(p-l-1)a_p|=|a_p|.$$
Similarly, the absolute value of the coefficient of $s^{p-1-l}$ in $u^l(u-\frac{dt}{ds})$ is at most $|a_p|$, since $p-1-l<0$, $|a_1|=1$ is maximal, and $a_p$ has maximal absolute value among the coefficients of the negative powers of $s$ in $u$ and $\frac{dt}{ds}$. Therefore, the absolute value of the coefficient of $s^{p-1}$ on the right hand side of \eqref{eq:goodcoordeq} is at most
$$|ra_p|<|a_p|,$$
which is a contradiction.
\end{proof}


Using good coordinates on an annulus $\CA_e$ with skeleton $e$, we can relate good formal coordinates at $p_e$ and $p_{e^\op}$. This is based on the following reduction construction assigning to a coordinate $t$ on $\CA_e$ compatible with the orientation a formal graded reduction $\tilt_x$ at $x=\ft(e)$. Choose a function $f$ defined in a neighborhood of $x$ such that $|f|_x=1$ and $\tilf$ has a simple zero at $p_e$. Then $f$ is a coordinate on the part of $\CA$ close enough to $x$ and $t=\sum_{-\infty}^\infty a_if^i$, where $|a_1|=\max_i|a_i|$ and $|a_i|<|a_1|$ for $i<1$. Set $b_i:=a_i/a_1$. Then it is easy to see that the reduction $$\tilt_x:=\left(\sum_{i=1}^\infty \widetilde{b_i}\tilf^i\right)\otimes\tila_1\in\widehat{\CO}_{C_x,p_e}\otimes_\tilk k^\gr$$ is independent of the choice of $f$.

Let now $X$ be a nice $k$-analytic curve equipped with a differential form $\omega$, and $\Gamma$ a skeleton of $X$ containing all zeros and poles of $\omega$. Let $e\in E(\Gamma)$ be a bounded oriented edge, and $x:=\ft(e)$, $y:=\fh(e)$ the corresponding points of type $2$. Set $n$ to be minus the slope of the level function along $e$. Consider the quotient $\CT_e$ of the torsor of good formal coordinates on $C_x$ at $p_e$ by the normal unipotent subgroup $U_n(\tilk)\lhd G_n(\tilk)$. If $n\ge 0$ then $U_n(\tilk)=0$, and otherwise $U_n(\tilk)=\tilk$. Then for any oriented edge $e$, the set $\CT_e$ is a $\tilk^\times$-torsor.

For any $r\in|k^\times|$, let $k_r\subset k$ is the subset of elements of norm $r$. The $k_r$ is a torsor under the natural action of the multiplicative subgroup $k_1\subset k^\times$, and we define $\CT_e^r$ to be the quotient of $k_r\times\CT_e$ by the anti-diagonal action $g(x,y):=(gx,g^{-1}y)$ of $k_1$. Set $\CT_e^\gr$ to be the disjoint union of $\CT_e^r$. Then $\CT_e^\gr$ is a graded torsor under the action of $(k^\times)^\gr$. We claim that the differential form $\omega$ induces an isomorphism $$\phi_e\:\CT_e^\gr\to\CT_{e^\op}^\gr$$ of graded $(k^\times)^\gr$-torsor with respect to the involution $(k^\times)^\gr\to (k^\times)^\gr$ given by $\lambda\mapsto\lambda^{-1}$, i.e., $\phi_e(\lambda t)=\lambda^{-1}\phi_e(t)$. Furthermore, $\phi_e$ (multiplicatively) shifts the grading by $\frac{\|\omega\|_y}{\|\omega\|_x}$, and $\phi_{e^\op}=\phi_e^{-1}$. This is the stacky reduction datum that should be added to the tropical reduction datum of $(X,\omega)$. It seems to be the algebraic analog of {\em prong matching} of \cite{BCGGM2}.

To describe $\phi_e$, consider the open annulus $\CA_e$, whose skeleton is supported on $e$, and pick a coordinate $s$ on $\CA_e$ compatible with the orientation of $e$. By  Proposition~\ref{goodprop}, there exists a good coordinate $t$ on $\CA_e$, and its graded reduction belongs to $\CT_e^\gr$. Furthermore, since $k^\times$ acts on the set of good coordinates on $\CA_e$, any graded element in $\CT_e^\gr$ is the graded reduction of a good coordinate on $\CA_e$. The map $\phi_e$ is then defined as follows: take a graded element $t_x\in\CT_e^\gr$ and lift it to a good coordinate $t$ on $\CA_e$. Then, $t^{-1}$ is a good coordinate on the oriented open annulus $\CA_{e^\op}$, and $\phi_e(t_x)$ is defined to be the graded reduction $\tilt_y$ of $t$ at $y$. Since the set of good coordinates compatible with the orientation on $\CA$ admits a natural transitive action of $G_n(\ko)$, it follows that $\phi_e(t_x)$ is well-defined. It is also clear from the construction that $\phi_e$ satisfies all the properties mentioned above.

\bibliographystyle{amsalpha}
\bibliography{biblio}
\end{document}